\documentclass[11pt]{amsart}
\usepackage{amssymb, latexsym}
\usepackage{graphicx}
\theoremstyle{plain}
\newtheorem{theorem}{Theorem}

\newtheorem {lemma}{Lemma}
\newtheorem{proposition}{Proposition}

\newtheorem*{1'}{Theorem 1-Bessel}
\newtheorem*{P2'}{Proposition 2-Bessel}
\newtheorem*{P3'}{Proposition 3-Bessel}
\newtheorem*{P4'}{Proposition 4-Bessel}
\newtheorem*{C1'}{Corollary 1-Bessel}

\newtheorem*{2'}{Theorem 2-Bessel}
\newtheorem*{3'}{Theorem 3-Bessel}

\theoremstyle{remark}

\newtheorem*{Remark 1}{Remark 1}
\newtheorem*{Remark 2}{Remark 2}
\newtheorem*{Remark 3}{Remark 3}
\newtheorem*{Remark 4}{Remark 4}

\numberwithin{equation}{section}

\begin{document}

\title [Comparison of Brownian jump and Brownian bridge resetting]
{Comparison of  Brownian jump and Brownian bridge resetting in  search for  Gaussian target on the line and in space}

\author{Ross G. Pinsky}


\address{Department of Mathematics\\
Technion---Israel Institute of Technology\\
Haifa, 32000\\ Israel}
\email{ pinsky@math.technion.ac.il}

\urladdr{https://pinsky.net.technion.ac.il/}

\date{}

\begin{abstract}

We consider a Brownian searcher with diffusion coefficient $D$ in $d$-dimensions, for $d=1,2,3$, that starts from the origin and searches for a random target with  a centered-Gaussian distribution. The searcher is also equipped with a resetting mechanism that resets the searcher back to the origin.
We consider three different resetting mechanisms. One is a Poissonian reset with rate $r$ whereby at the reset time the searcher instantaneously jumps back to the origin,  one is a periodic  reset with period $T$ whereby at the reset time  the searcher
 instantaneously jumps back to the origin, and one is a Brownian bridge reset with period $T$, whereby the Brownian motion is conditioned to return to the origin at time $T$. Unlike the first two search processes, this last one
 has continuous paths.
 For $d=1$ and $d=3$, we obtain analytic formulas for the expected time to locate the random target, and minimize them over $r$ or $T$, as the case may be. In one dimension, this expected time scales
 as $\frac{\sigma^2}D$, which is not surprising, but  in three dimensions we obtain the anomalous scaling $\frac{\sigma^3}D$. We compare the relative efficiencies of the three search processes. In two dimensions, we show that
 the expected time scales as $\frac{\sigma^2}D$ for the Poissonian reset mechanism.

\end{abstract}

\maketitle
\section{Introduction and Statement of Results}\label{intro}
The use of resetting in  search problems is
 a  common phenomenon in various contexts.  For example, in everyday life,  one might be searching for some target, such as a face in a crowd or a misplaced object. After having searched unsuccessfully for a while,
  there is   a tendency to return to the starting point  and begin the search anew.
Other contexts where search problems  frequently involve resetting include  animal foraging \cite{BC,VdRS} and internet search algorithms.

Over the past decade or so, a variety of  stochastic processes with resetting have attracted much attention. See \cite{EMS} for a rather comprehensive, recent overview.  Prominent among such processes is the  diffusive search process with instantaneous resetting, which we now describe.
Consider a   target located at  $a\in \mathbb{R}^d$, $d\ge1$,
 and consider  a search process that  sets off from the origin and performs  $d$-dimensional Brownian motion with diffusion coefficient $D>0$, which is fixed once and for all.
Throughout the paper, we  suppress the dependence on $D$ in our notation.
The search process  is also equipped with   an exponential clock with  rate $r$,  so that
if it has  failed to locate the target by the time the clock rings, then its position is instantaneously reset to the origin and it continues  its search anew independently with the same  rule.
See, for example, \cite{EM1,EM3} for details of the construction of the process.
We consider $r$ as a parameter that can be varied.
In dimension one, the target is considered ``located'' when the process
hits the point $a$. In dimensions two and higher, since the probability of a Brownian motion ever hitting a particular point is zero, one chooses an $\epsilon_0>0$ which is fixed once and for all,  and the target is considered ``located''
when the process hits the $\epsilon_0$-ball centered at $a$.

Denote the search process by $X^{(d;r)}(\cdot)$ and let $P_0^{(d;r)}$  and  $E_0^{(d;r)}$ denote probabilities and  expectations for the process starting from 0.
Let
\begin{equation}\label{stoppingtime}
\tau_a=\begin{cases}\inf\{t\ge0:X^{(1;r)}(t)=a\},\ d=1;\\ \inf\{t\ge0:|X^{(d;r)}(t)-a|\le\epsilon_0\},\ d\ge 2\end{cases}
\end{equation}
denote the time at which a target at  $a\in\mathbb{R}^d$ is located.
Throughout the paper, we suppress the dependence on $\epsilon_0$ in the notation. For the above search process, as well as for other related models,
quite a number of papers have investigated  a number of phenomena, in particular,  the expected time to reach the target, $E_0^{(d;r)}\tau_a$,  the probability of not locating the target for
large time, $P_0^{(d;r)}(\tau_a>t)$ for large $t$, and  the stationary probability measure for the process.
 See, for example, \cite{EM1, EM2,EM3, EMM, P, NG, dMMT, RR}.


The resetting in the above model is of course discontinuous---at the ring of the exponential clock, the  search process instantaneously jumps back to its initial position at the origin.
In certain applications, this may be a realistic assumption, but in many others, it is more realistic to consider a type of resetting for which the search process remains continuous.
For example, while the instantaneous jump model might be reasonable for an internet search,  a continuous type of resetting would be more realistic for animal foraging.
A number of very recent papers have addressed this issue, using  deterministic return processes, sometimes with constant velocity in the direction of the reset point, and sometimes with other regimes
\cite{BS}, \cite{MCM},\cite{PKR}, \cite{B}.
In these papers, in order to make the problem more tractable mathematically, the target is not allowed to be discovered during the return process.
These papers study, in particular, the stationary  probability distribution for these processes as well as the expected time to reach the target at $a\in\mathbb{R}^d$.

In this paper, we  introduce a continuous search process with   resetting via the Brownian bridge.
 The Brownian bridge with time interval $T$, which we describe in more detail below, is a Brownian motion conditioned to return to its starting point at time $T$.
 The search process, which we call the Brownian bridge reset search process with time interval $T$, performs one Brownian bridge with time interval $T$ after the other, until it locates the target.
 Unlike the continuous models in the papers cited in the above paragraph, this model  does not have a search regime with its own parameters, followed by a return regime with its own parameters.
Rather, there is one seamless process which searches and returns with two  parameters throughout; namely, the diffusion coefficient $D$ and the time interval $T$.
Furthermore, unlike in the above models, the  target may be located at any time.
(We note that a recent paper \cite{DMS} introduced a hybrid version of the above two search processes in the one-dimensional case. The  process in that paper is the search process
with instantaneous resetting, conditioned to return to the origin at time $T$. The authors study various properties of this process on the time interval $[0,T]$.)

Most of the papers in the literature deal with the expected hitting time of a fixed target $a$, rather than a random target, although  we note that the early paper \cite{EM2} does consider random targets.
Of course, once one has a formula for the expected hitting time of  a fixed target, the formula for a random target is obtained simply by integration. However, we are interested in optimizing the parameters to obtain the
smallest possible hitting time. In order to understand the scaling and to obtain numerically the optimal hitting time, one needs to express the  integrated expected hitting time in a reasonably nice closed form.
In this paper we will study the expected hitting time of a random target
$a\in\mathbb{R}^d$ that is distributed according to a Gaussian distribution centered at the origin, the point to which the search process is reset.
Our choice of centered Gaussian distributions has been dictated by a combination of what would be interesting and natural, and by what will lend itself to closed form formulas.
(We note that for the case of instantaneous resetting,  a recent paper \cite{Pin23} studied the behavior of the probability of not locating a random target for large time, where the target comes from a rather wide family of symmetric distributions, including Gaussian distributions.)

We wish to
 compare the  efficiency
 of  the Brownian bridge reset search process,  with its parameters $D$ and $T$,  to the instantaneous reset  search process described above, with its parameters
 $D$ and $r$, where efficiency is measured by the infimum of the expected value of the search time, the infimum being taken over $T$ and over $r$ respectively for the two processes.
This leads us naturally  to consider also  a third search process, as we now explain.

 The instantaneous reset search process and the Brownian bridge reset search process are both Markov processes, but there is one fundamental difference between them;
 namely the instantaneous process is time-homogeneous while the Brownian bridge reset search process is time-inhomogeneous, with cyclic time-inhomogeneity of period $T$.
 It seems  that there is no way around the time-inhomogeneity of the Brownian bridge reset search process. However, we can easily change the instantaneous process from
 time homogeneous to time-inhomogeneous with cyclic time-inhomogeneity of period $T$. Indeed, instead of using an exponential clock with rate $r$ to determine when the process jumps back to its starting point, we simple let it jump back to its starting point every $T$ units of time.

From now on, in order to fall in line with rather prevalent terminology in the literature,  we will refer to the original instantaneous reset search process, with parameters $D$ and $r$, as the Poissonian instantaneous reset search process, and we will refer to the instantaneous reset search process described in the previous paragraph, with parameters $D$ and $T$, as the periodic instantaneous reset search process.

Denote the Brownian bridge reset search process by $X^{\text{bb},d;T}(\cdot)$ and let
 $P_0^{\text{bb},d;T}$ and $E_0^{\text{bb},d;T}$ denote probabilities and expectations for  process  starting from 0.
Denote the periodic  instantaneous reset search process by   $X^{d;T}(\cdot)$ and let
$P_0^{d;T}$ and $E_0^{d;T}$ denote probabilities and expectations for the process starting from 0.
We have already established notation for the Poissonian instantaneous reset search process.
Note that the only difference in notation between the Poissonian and the periodic  instantaneous reset search processes is the use of $T$ versus $r$.
For the
Brownian bridge reset search process and the periodic  instantaneous reset search process, we
 use the same notation, $\tau_a$, as was used for the Poissonian reset search process for the time to locate the target.
 Thus, for these two processes, $\tau_a$ is defined as in \eqref{stoppingtime}, but with
 $X^{d;r}$ replaced by $X^{\text{bb},d;T}$ or by $X^{d;T}$.
Before we state our results, we define rigorously the two search processes $X^{\text{bb},d;T}(\cdot)$ and  $X^{d;T}(\cdot)$.

\begin{figure}\label{fig1}
\includegraphics[scale=.2]{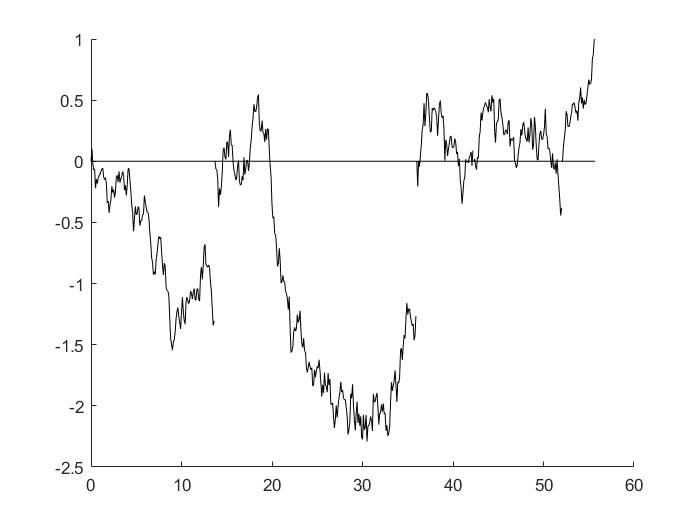}
\caption{Poissonian instantaneous reset; $r=\frac1{10},\ a=1$ }
\end{figure}

\begin{figure}\label{fig2}
\includegraphics[scale=.2]{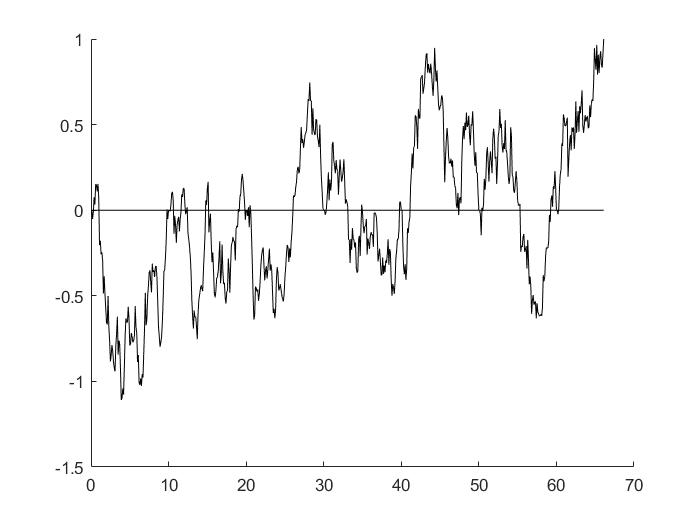}
\caption{Brownian bridge reset; $T=10, \ a=1$ }
\end{figure}

\begin{figure}\label{fig3}
\includegraphics[scale=.2]{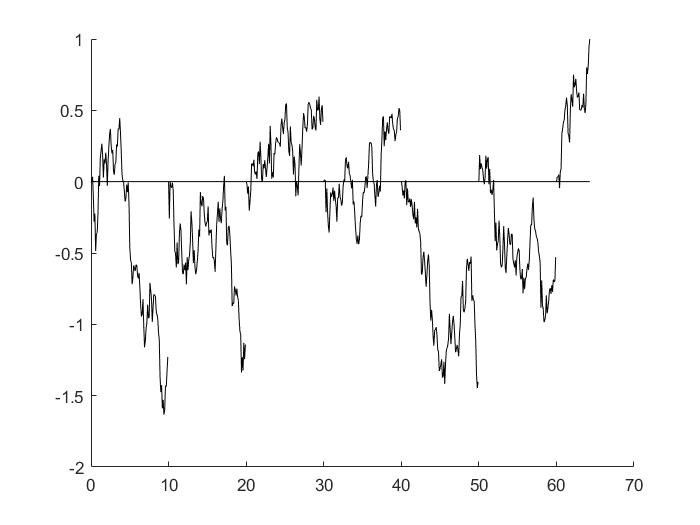}
\caption{Periodic instantaneous reset; $T=10,\ a=1$ }
\end{figure}

 As above, fix  once and for all the diffusion coefficient $D>0$.
Recall that the one-dimensional Brownian bridge with bridge time interval $T$ is the one-dimensional Brownian motion conditioned to be at the origin at time $T$. For background on this process, see
for example, \cite{KS,RW}.
 As is well-known \cite{KS, RW}, a one-dimensional Brownian bridge
 with bridge time interval $T$ and diffusion coefficient $D$ can be represented as $\big(W(t)-\frac tTW(T)\big),\ 0\le t\le T$, where $W(\cdot)$ is a one-dimensional Brownian motion with diffusion coefficient $D$. The $d$-dimensional Brownian bridge with bridge time interval $T$ and diffusion coefficient $D$ is the process in $\mathbb{R}^d$ whose components are independent one-dimensional Brownian bridges with bridge time interval $T$ and diffusion coefficient $D$.
For $T>0$, let
$\{B_n^{\text{bb},d;T}(t),0\le t\le T\}_{n=1}^\infty$ be a sequence of independent  $d$-dimensional Brownian bridges with bridge time interval $T$ and diffusion coefficient $D$.
 Define the Brownian bridge reset search process $X^{\text{bb},d;T}(\cdot)$ by
\begin{equation}\label{Xbb}
X^{\text{bb},d;T}(t)=B_n^{\text{bb},d;T}(t-nT),\ t\in[nT,(n+1)T),\ n=0,1,2,\cdots.
\end{equation}

Similarly, let $\{B_n^{d}(t),0\le t\le T\}_{n=1}^\infty$ be a sequence of independent $d$-dimensional Brownian motions with diffusion coefficient $D$.
Define the periodic  instantaneous
reset search process $X^{d;T}(\cdot)$ by
\begin{equation}\label{Xinhomoginstant}
X^{d;T}(\cdot)=B_n^d(t-nT),\ t\in[nT,(n+1)T),\ n=0,1,2,\cdots.
\end{equation}

For the processes $X^{\text{bb},d;T}$ and $X^{d;T}$,
we consider $T$, the time interval between resets, to be a parameter that can be varied, just as we consider the rate $r$ of the exponential clock  a parameter that can be varied for the
process $X^{d;r}$.
Figures 1-3 present a simulation for each of the three search processes with target $a=1$.

Let $\mu^{\text{Gauss},d}_{\sigma^2}$ denote the centered Gaussian distribution on $\mathbb{R}^d$ with variance $\sigma^2$.
Under the three search processes we have defined, we wish to calculate the expected time to locate the target; namely,
\begin{equation}\label{all3}
\begin{aligned}
&\int_{\mathbb{R}^d}E_0^{d;r}\tau_a\thinspace\mu^{\text{Gauss},d}_{\sigma^2}(da), \  \ \int_{\mathbb{R}^d}E_0^{\text{bb},d;T}\tau_a\thinspace\mu^{\text{Gauss},d}_{\sigma^2}(da)\\
&\text{and}\
\int_{\mathbb{R}^d}E_0^{d;T}\tau_a\thinspace\mu^{\text{Gauss},d}_{\sigma^2}(da).
\end{aligned}
\end{equation}
Then we compare the efficiency of the processes by
calculating the infimum (over $T$ or $r$ as appropriate)  of the above expected time to locate the target.

Of course, the first step is to evaluate
$E_0^{d;r}\tau_a$, $E_0^{\text{bb},d;T}\tau_a$ and $E_0^{d;T}\tau_a$, for fixed $a$. Expressions for $E_0^{d;r}\tau_a$ and $E_0^{d;T}\tau_a$ appear in the literature, as will be noted in sections
\ref{exp-a-timehomog} and \ref{exp-a-inhomog}.
It turns out that in dimensions $d=1$ and $d=3$, closed form formulas or close to closed form formulas exist for  these expressions. We will show in this paper that the same is true for
$E_0^{\text{bb},d;T}\tau_a$.
A reasonably closed form formula in dimension $d=2$ is known for the first  of these three expressions. As will be seen below in sections \ref{exp-a-bb} and \ref{exp-a-inhomog},
in order to get reasonably nice formulas for the second and third expressions, one needs to
have a reasonably nice formula for the probability density  of $\tau_a$ under standard $d$-dimensional Brownian motion (actually, sub-density in dimensions $d\ge3$, since in these dimensions $\tau_a=\infty$
with positive probability).
One can obtain this density (or sub-density, as the case may be)
from  Theorem 2.2 in \cite{HM}. Whereas the formula is nice in dimensions one and three, it is quite  unwieldy in dimension two.
In light of this, we are able
to calculate in a reasonably  explicit way all three expressions in \eqref{all3} in dimensions $d=1$ and $d=3$, but only the first of these three expressions in dimension $d=2$.
Thus, in the sequel, these will be the expressions we study.

The results  for the above expectations for fixed $a$ will be presented in sections \ref{exp-a-timehomog}--\ref{exp-a-inhomog}. It turns out that
$E_0^{d;r}\tau_a$ grows exponentially in $a$,
thus the corresponding expression in \eqref{all3} is finite for all choices of $D,r$ and $\sigma^2$. However,
 each of the other two expectations grows on the order $e^{C(D.T)\thinspace a^2}$ for an appropriate constant $C(D,T)$, and as will be seen, the corresponding expressions in \eqref{all3} are finite
  only for appropriate values of $D, T$ and $\sigma^2$.
(We note that \cite{Pin20} considered a Poissonian instantaneous reset process in one-dimension in which the exponential reset clock is replaced by a spatially dependent clock. By appropriate choice of the spatial dependence, it is shown there that the expected value of $\tau_a$ can be made to grow on the order
$|a|^{2+\delta}$, for any $\delta>0$.)

As will  be seen, with regard to the parameter $D$, the diffusion coefficient of the search process, and the parameter
$\sigma^2$, the variance of the target distribution,
in dimension $d=1$, all three expressions in \eqref{all3} scale in a standard way; namely as $\frac{\sigma^2}D$. However, in dimension $d=3$, we obtain the
  anomalous scaling $\frac{\sigma^3}D$.
This, of course,  leads one to wonder what scaling occurs in dimension $d=2$. We will show that in dimension $d=2$, the first expression  in
\eqref{all3} has the standard scaling obtained in the one-dimensional case.
Indeed, this renders the scaling obtained in the three-dimensional case all the more anomalous.

We now state our main results; namely, the explicit calculations of the expressions in \eqref{all3} and the corresponding infima over $r$ or $T$, as appropriate, for dimensions $d=1$ and
$d=3$, and also  for dimension $d=2$ in the case of  the first of the three expressions in \eqref{all3}.
We begin with the
one-dimensional case.
 Here is the result for the Poissonian  instantaneous reset search  process.
\begin{theorem}\label{resetGauss}
\begin{equation}\label{expGauss}
\int_{\mathbb{R}}\big(E_0^{(1;r)}\tau_a\big)\mu^{\text{Gauss},1}_{\sigma^2}(da)=\frac1r\big(2e^{\frac{r\sigma^2}D}\int_{-\sqrt{\frac{2r}D}\sigma}^\infty \frac{e^{-\frac{x^2}2}}{\sqrt{2\pi}}dx-1\big).
\end{equation}
Equivalently, writing  $r=\frac D{\sigma^2}s$, with $s>0$,
\begin{equation}\label{equiv}
\int_{\mathbb{R}}\big(E_0^{(1;\frac D{\sigma^2}s)}\tau_a\big)\mu^{\text{Gauss},1}_{\sigma^2}(da)=\frac{\sigma^2}D\Big(\frac{2e^s \int_{-\sqrt{2s}}^\infty\frac{e^{-\frac{x^2}2}}{\sqrt{2\pi}}dx-1}s\Big).
\end{equation}
One has
\begin{equation}\label{graphcalc}
\inf_{r>0}\int_{\mathbb{R}}\big(E_0^{(1;r)}\tau_a\big)\mu^{\text{Gauss},1}_{\sigma^2}(da)\approx 3.548\frac{\sigma^2}D\ \ \ \text{with the infimum attained at} \ r\approx 0.491\frac D{\sigma^2}.
\end{equation}
\end{theorem}
Here is the  corresponding result for the  Brownian bridge reset search process.
\begin{theorem}\label{bridgeGauss}
\begin{equation}\label{expbridgeGauss}
\int_{\mathbb{R}}\big(E_0^{\text{bb},1;T}\tau_a\big)\mu^{\text{Gauss},1}_{\sigma^2}(da)=\begin{cases} T(\frac{DT}{DT-4\sigma^2})^\frac12-T+\frac{T\sigma}{\sqrt{DT}+2\sigma},\ T>\frac{4\sigma^2}D;\\ \infty,\ T\le \frac{4\sigma^2}D.\end{cases}.
\end{equation}
Equivalently, writing $T=\frac{\sigma^2}D\mathcal{T}$, with $\mathcal{T}>0$,
\begin{equation}\label{equivagain}
\int_{\mathbb{R}}\big(E_0^{\text{bb},1;\frac{\sigma^2}D\mathcal{T}}\tau_a\big)\mu^{\text{Gauss},1}_{\sigma^2}(da)=
\begin{cases}\frac{\sigma^2}D\Big(\mathcal{T}(\frac{\mathcal{T}}{(\mathcal{T}-4)})^\frac12-\mathcal{T}+\frac {\mathcal{T}}{2+\sqrt{\mathcal{T}}}\Big), \mathcal{T}>4;\\ \infty, \mathcal{T}\le 4.\end{cases}
\end{equation}
One has
\begin{equation}\label{graphcalcagain}
\inf_{T>0}\int_{\mathbb{R}}\big(E_0^{\text{bb},1;T}\tau_a\big)\mu^{\text{Gauss},1}_{\sigma^2}(da)\approx4.847\frac{\sigma^2}D\ \text{with the infimum attained at}\ T\approx 10.136\frac{\sigma^2}D.
\end{equation}
\end{theorem}
And here is the corresponding result for the periodic  instantaneous reset search process.
\begin{theorem}\label{inhomog1d}
\begin{equation}\label{expGaussinhomog1}
\begin{aligned}
&\int_{\mathbb{R}}\big(E_0^{(1;T)}\tau_a\big)\mu^{\text{Gauss},1}_{\sigma^2}(da)=\\
&\begin{cases}
\frac2{\sqrt{2\pi}}\int_0^\infty\frac{\int_0^T\frac1{t^{\frac12}}e^{-\frac{\sigma^2x^2}{2Dt}}dt}{\int_0^T\frac1{t^{\frac32}}e^{-\frac{\sigma^2x^2}{2Dt}}dt}\thinspace  e^{-\frac{x^2}2}dx+
T\Big(\frac{2\sqrt D}{\sigma}\int_0^\infty \frac1{\int_0^T\frac x{t^\frac32}e^{-\frac{\sigma^2x^2}{2Dt}}dt}\thinspace e^{-\frac{x^2}2} dx-1\Big),\ T>\frac{\sigma^2}D;\\
\infty,\ T\le\frac{\sigma^2}D.\end{cases}
\end{aligned}
\end{equation}
Equivalently, writing  $T=\frac{\sigma^2}D\mathcal{T}$, with $\mathcal{T}>0$,
\begin{equation}\label{equivagainn}
\begin{aligned}
&\int_{\mathbb{R}}\Big(E_0^{(1;\frac{\sigma^2}D\mathcal{T})}\tau_a\big)\mu^{\text{Gauss},1}_{\sigma^2}(da)=\\
&\begin{cases}\frac{\sigma^2}D\Big(\frac2{\sqrt{2\pi}}\int_0^\infty
\frac{\int_0^\mathcal{T}\frac1{s^\frac12}e^{-\frac{x^2}{2s}}ds}{\int_0^\mathcal{T}\frac1{s^\frac32}e^{-\frac{x^2}{2s}}ds}\thinspace e^{-\frac{x^2}2} dx+2\mathcal{T}\int_0^\infty\frac1{\int_0^\mathcal{T}\frac x{s^\frac32}
e^{-\frac{x^2}{2s}}ds}\thinspace e^{-\frac{x^2}2}dx-\mathcal{T}\Big),\ \mathcal{T}>1,\\ \infty,\ \mathcal{T}\le1.\end{cases}
\end{aligned}
\end{equation}
One has
\begin{equation}\label{graphcalcagainn}
\inf_{T>0}\int_{\mathbb{R}}\big(E_0^{(1;T)}\tau_a\big)\mu^{\text{Gauss},1}_{\sigma^2}(da)\approx 3.35\frac{\sigma^2}D\ \ \ \text{with the infimum attained at} \ T\approx 2.82\frac{\sigma^2}D.
\end{equation}
\end{theorem}
\medskip

\bf\noindent Conclusion for the one-dimensional case.\rm\
From Theorems \ref{resetGauss}--\ref{inhomog1d}, it follows that the appropriate scaling unit for the resetting rate $r$ in the case
of the Poissonian instantaneous reset search process is $\frac D{\sigma^2}$, and the appropriate scaling unit for the  time interval $T$ in the case of the Brownian bridge reset search process
and the periodic  instantaneous reset search process is
$\frac{\sigma^2}D$. In all of these cases, the corresponding expected time to locate the target is an appropriate constant times $\frac{\sigma^2}D$. This is a natural and expected scaling.

Comparing the two time-inhomogeneous processes, one with instantaneous jump resetting and one with continuous resetting via the Brownian bridge, one sees from
\eqref{graphcalcagain} and \eqref{graphcalcagainn} that the
optimal expected time to locate the target is about 37 percent longer for the continuous search process than for the one with instantaneous jumps.
It is not surprising that the Brownian bridge reset search process is the less efficient of the two, as it will never locate the target on its way back to the origin since it is going through territory it has already explored.
From \eqref{graphcalc} and \eqref{graphcalcagainn} one sees that the optimal expected time to locate the target is about 6 percent longer for the  Poissonian instantaneous reset search process than for the periodic  instantaneous reset search process. We note that in the context of a search for a \it deterministic \rm\ target, it was  shown in \cite{PR}  that the optimal expected time to locate the target for the periodic instantaneous reset  search process is always less than for the Poissonian instantaneous search process (and also less than for a generalization of the Poissonian search process where the jump rate is time dependent).
We see here that this advantage of the periodic case over the Poissonian case continues to hold for a centered Gaussian target distribution.
As will be  seen, the relative efficiencies of the three processes  in the three dimensional case are quite different from the one-dimensional case, and in particular, the above-noted phenomenon no longer holds.

The expected value of the time interval between resets in the Poissonian instantaneous reset search process is $\frac1r$. Thus, for the optimal value of $r$ in \eqref{graphcalc}, the expected value of the time interval between resets is about
$2.04\frac{\sigma^2}D$.
From \eqref{graphcalc} and \eqref{graphcalcagainn}, one sees that the optimal value of $T$ for the Brownian bridge reset search process  is about 3.6 times as large as that for the
periodic  instantaneous reset search process

\bigskip

We now turn to the corresponding results in three dimensions.
As will be seen in \eqref{3r}, \eqref{expectonept3d} and \eqref{expainhomog3form},
the expected hitting time $\tau_a$ of a point $a\in\mathbb{R}^3$ with $|a|>\epsilon_0$  depends on $\epsilon_0$, and diverges on the order
$\frac1{\epsilon_0}$ as $\epsilon_0\to0$, for all three of the search processes. Thus, in the case $d=3$, it is appropriate to multiply each of the expressions in \eqref{all3} by $\epsilon_0$ and then consider the  limit
as $\epsilon_0\to0$.
Here is the result for the Poissonian instantaneous reset search process.
\begin{theorem}\label{resetGauss3}
\begin{equation}\label{resetGauss3form}
\begin{aligned}
&\lim_{\epsilon_0\to0}\epsilon_0\int_{\mathbb{R}^3}\big(E_0^{(3;r)}\tau_a\big)\mu^{\text{Gauss},3}_{\sigma^2}(da)=
\frac{2\sigma}{\sqrt{2\pi}\thinspace r}\int_0^\infty x^3e^{\sqrt{\frac rD}\thinspace\sigma x}e^{-\frac{x^2}2}dx.
\end{aligned}
\end{equation}
Equivalently, writing $r=\frac{D}{\sigma^2}s$, with $s>0$,
\begin{equation}\label{resetGauss3formagain}
\lim_{\epsilon_0\to0}\epsilon_0\int_{\mathbb{R}^3}\big(E_0^{(3;r)}\tau_a\big)\mu^{\text{Gauss},3}_{\sigma^2}(da)=\frac{\sigma^3}D\Big(\frac2{\sqrt{2\pi}s}\int_0^\infty x^3e^{\sqrt s\thinspace x}e^{-\frac{x^2}2}dx\Big).
\end{equation}
One has
\begin{equation}\label{graphiccalc3}
\begin{aligned}
&\inf_{r>0}\lim_{\epsilon_0\to0}\epsilon_0\int_{\mathbb{R}^3}\big(E_0^{(3;r)}\tau_a\big)\mu^{\text{Gauss},3}_{\sigma^2}(da)\approx13.09\frac{\sigma^3}D,\\
&\text{with the infinmum attained at }\ r\approx 0.738\frac{D}{\sigma^2}.
\end{aligned}
\end{equation}

\end{theorem}
Here is the  corresponding result for the  Brownian bridge reset search process.
\begin{theorem}\label{bridgeGauss3}
\begin{equation}\label{expbridgeGauss3}
\lim_{\epsilon_0\to0}\epsilon_0\int_{\mathbb{R}^3}\big(E_0^{\text{bb},3;T}\tau_a\big)\mu^{\text{Gauss},3}_{\sigma^2}(da)=\begin{cases}
\frac{2T^3D^2\sigma}{\sqrt{2\pi}(DT-4\sigma^2)^2},\ T>\frac{4\sigma^2}D;\\ \infty,\ T\le \frac{4\sigma^2}D.\end{cases}.
\end{equation}
Equivalently, writing $T=\frac{\sigma^2}D\mathcal{T}$,
\begin{equation}\label{expbridgeGauss3again}
\lim_{\epsilon_0\to0}\epsilon_0\int_{\mathbb{R}^3}\big(E_0^{\text{bb},3;T}\tau_a\big)\mu^{\text{Gauss},3}_{\sigma^2}(da)=\begin{cases}
\frac{\sigma^3}D\Big(\frac{2\mathcal{T}^3}{\sqrt{2\pi}(\mathcal{T}-4)^2}\Big),\ \mathcal{T}>4;\\ \infty,\ \mathcal{T}\le 4.\end{cases}
\end{equation}
One has
\begin{equation}\label{graphiccalc3again}
\begin{aligned}
&\inf_{T>0}\lim_{\epsilon_0\to0}\epsilon_0\int_{\mathbb{R}^3}\big(E_0^{\text{bb},3;T}\tau_a\big)\mu^{\text{Gauss},3}_{\sigma^2}(da)\approx
21.54\frac{\sigma^3}D,\\
&\text{with the infimum attained at}\ T=12.00\frac{\sigma^2}D.
\end{aligned}
\end{equation}
\end{theorem}
And  here is the corresponding result for the periodic  instantaneous reset search process.

\begin{theorem}\label{inhomog3d}
\begin{equation}\label{expGaussinhomog3d}
\begin{aligned}
&\lim_{\epsilon_0\to0}\epsilon_0\int_{\mathbb{R}^3}\big(E_0^{(3;T)}\tau_a\big)\mu^{\text{Gauss},3}_{\sigma^2}(da)=
\begin{cases} 2\sqrt DT\int_0^\infty \frac1{\int_0^T\frac1{t^{\frac32}}e^{-\frac{\sigma^2 x^2}{2Dt}}dt}\thinspace x^2e^{-\frac{x^2}2}dx,\ T>\frac{\sigma^2}D,\\
\infty,\  T\le \frac{\sigma^2}D.\end{cases}
\end{aligned}
\end{equation}

Equivalently, writing $T=\frac{\sigma^2}D\mathcal{T}$, with $\mathcal{T}>0$,
\begin{equation}\label{equivinhomog3d}
\lim_{\epsilon_0\to0}\epsilon_0\int_{\mathbb{R}^3}\big(E_0^{(3;T)}\tau_a\big)\mu^{\text{Gauss},3}_{\sigma^2}(da)=\begin{cases}
\frac{\sigma^3}D\Big(2\mathcal{T}\int_0^\infty
 \frac1{\int_0^\mathcal{T}\frac1{s^{\frac32}}e^{-\frac{x^2}{2s}}ds}\thinspace x^2e^{-\frac{x^2}2}\Big)dx,\ \mathcal{T}>1,\\
\infty, \ \mathcal{T}\le 1.
\end{cases}
\end{equation}
One has
\begin{equation}\label{graphiccalc3againn}
\begin{aligned}
&\inf_{T>0}\lim_{\epsilon_0\to0}\epsilon_0\int_{\mathbb{R}^3}\big(E_0^{(3;T)}\tau_a\big)\mu^{\text{Gauss},3}_{\sigma^2}(da)\approx22.775\frac{\sigma^3}D,\\
&\text{with the infinmum attained at }\ T\approx 4.13\frac{\sigma^2}D.
\end{aligned}
\end{equation}

\end{theorem}

\bf\noindent Conclusion for the three-dimensional case.\rm\
From Theorems \ref{resetGauss3}--\ref{inhomog3d},
it follows that,  as in the one-dimensional case,  the appropriate scaling unit for the resetting rate $r$ in the case
of the Poissonian instantaneous reset search process is $\frac D{\sigma^2}$, and the appropriate scaling unit for the  time interval $T$ in the case of the Brownian bridge reset search process
and the periodic  reset search process is
$\frac{\sigma^2}D$.
 \it However, for all of these processes, the corresponding expected time to locate the target is an appropriate constant times  $\frac{\sigma^3}D$, as opposed to $\frac{\sigma^2}D$ in the one-dimensional case.\rm\
This seems to us to be an anomalous and surprising scaling.

Comparing the two time-inhomogeneous processes, one with instantaneous jump resetting and one with continuous resetting via the Brownian bridge, one sees from
\eqref{graphiccalc3again} and \eqref{graphiccalc3againn} that the
optimal expected time to locate the target is about 6 percent longer for the one with instantaneous jumps than for the continuous one,
which is completely different than in the one-dimensional case. Perhaps the explanation for this is that, unlike in the one-dimensional case,  in the three-dimensional case the Brownian bridge reset search process
can explore new territory on its way back to the origin.
From \eqref{graphiccalc3}, \eqref{graphiccalc3again} and \eqref{graphiccalc3againn}, one sees that the Poissonian instantaneous reset search process is by far the most efficient
one of the three. The expected time to locate the target is about 65 percent and 74 percent longer for the other two processes.
Note that the
periodic  instantaneous reset search process fares the worst in  three dimensions while faring the best in one dimension.

\medskip

In light of the anomalous
scaling in three dimensions, it is natural to wonder what occurs in two dimensions.
As noted above, we are able to handle  the two-dimensional Poissonian instantaneous reset search process. Theorem \ref{2dimcase} below shows that the scaling in this case is the natural $\frac{\sigma^2}D$ scaling.
As will be seen in the proof of  Theorem \ref{2dimcase}  in section \ref{pfthm7},
$E_0^{(2;T)}\tau_a$,  with $|a|>\epsilon_0$,  depends on $\epsilon_0$, and diverges on the order $|\log\epsilon_0|$ as $\epsilon_0\to0$.
Thus, in the case $d=2$, we multiply the first expression in \eqref{all3} by $\frac1{|\log\epsilon_0|}$ and then consider the limit as $\epsilon_0\to0$.
\begin{theorem}\label{2dimcase}
\begin{equation}\label{2dimformula}
\lim_{\epsilon\to0}\frac1{|\log\epsilon_0|}\int_{\mathbb{R}^2}(E_0^{(2;T)}\tau_a)\mu^{\text{Gauss},2}_{\sigma^2}(da)=\frac1r\int_0^\infty \frac1{K_0(\sqrt{\frac rD}\thinspace \sigma x)}xe^{-\frac{x^2}2}dx,
\end{equation}
where $K_0$ is the modified Bessel function of the second kind of order zero.

\noindent Equivalently, writing $r=\frac D{\sigma^2}s$,
\begin{equation}\label{equiv2dim}
\lim_{\epsilon\to0}\frac1{|\log\epsilon_0|}\int_{\mathbb{R}^2}(E_0^{(2;T)}\tau_a)\mu^{\text{Gauss},2}_{\sigma^2}(da)=\frac{\sigma^2}D\big(\frac1s\int_0^\infty\frac1{K_0(\sqrt s\thinspace x)}\thinspace xe^{-\frac{x^2}2}dx\big).
\end{equation}
One has
\begin{equation}\label{graphcalc2}
\inf_{r>0}\lim_{\epsilon\to0}\frac1{|\log\epsilon_0|}\int_{\mathbb{R}^2}(E_0^{(2;T)}\tau_a)\mu^{\text{Gauss},2}_{\sigma^2}(da)\approx 4.77\frac{\sigma^2}D,
\end{equation}
with the infimum attained at    $r\approx0.713\frac D{\sigma^2}$.
\end{theorem}
\noindent \bf Remark.\rm\ Comparing \eqref{graphcalc} and \eqref{graphcalc2}, one sees that for the Poissonian instantaneous reset search process, the expected time to locate the target is about
34 percent longer in the two-dimensional case than in the one-dimensional case. Of course, there is no sense in making a comparison with the three-dimensional case since the scaling there is different.
\medskip

As already mentioned, the expectation of $\tau_a$ for fixed $a$, for all of the various cases is treated in sections \ref{exp-a-timehomog}--\ref{exp-a-inhomog}. These results are then used  to prove
Theorems \ref{resetGauss}-\ref{2dimcase}. The proofs of Theorems \ref{resetGauss}--\ref{inhomog1d} are given in section \ref{pfthms1-3}, the proofs of Theorems \ref{resetGauss3}--\ref{inhomog3d} are given in
section \ref{pfthms4-6} and the proof of Theorem \ref{2dimcase} is given in section \ref{pfthm7}.

\section{Calculating the expected value of $\tau_a$ for fixed $a$ in the Poissonian instantaneous reset case}\label{exp-a-timehomog}
For the Poissonian instantaneous reset case, the calculation we need appears in the literature.
In dimension $d=1$,
\begin{equation}\label{r}
E_0^{(1;r)}\tau_a=\frac1r\big(e^{\sqrt{\frac{2r}D}\thinspace |a|}-1\big),\ a\in\mathbb{R},
\end{equation}
and in dimensions $d\ge2$,
\begin{equation}\label{dr}
E_0^{(d;r)}\tau_a=\frac1r\Big((\frac{\epsilon_0}{|a|})^{1-\frac d2}\frac{K_{1-\frac d2}(\sqrt \frac rD\thinspace\epsilon_0)}{K_{1-\frac d2}(\sqrt\frac rD\thinspace |a|)}-1\Big),\ |a|>\epsilon_0,
\end{equation}
where $K_\nu$ denotes the modified Bessel function of the second kind of order $\nu$  \cite{EM1,EM3}.
Moreover,  it is known that
$$
K_{-\frac12}(y)=(\frac\pi{2y})^\frac12e^{-y}.
$$
Thus, from \eqref{dr} we have
\begin{equation}\label{3r}
E_0^{(3;r)}\tau_a=\frac1r\Big(\frac{|a|}{\epsilon_0}\thinspace e^{\sqrt{\frac rD}(|a|-\epsilon_0)}-1\Big),\ |a|>\epsilon_0.
\end{equation}

\section{Calculating the expected value of $\tau_a$ for fixed $a$ in the Brownian bridge reset case}\label{exp-a-bb}

We consider the Brownian bridge reset search process for dimensions $d=1$ and $d=3$.
For the one-dimensional case, we have the following result.
\begin{proposition}\label{expectationa}
\begin{equation}\label{expectonept}
\begin{aligned}
&E_0^{\text{bb},1;T}\tau_a=T(e^{\frac{2a^2}{DT}}-1)+|a|e^{\frac{2a^2}{DT}}\int_0^T\frac{e^{-\frac{a^2}{2Dt(1-\frac tT)}}}{\sqrt{2\pi Dt(1-\frac tT)}}dt=\\
&T(e^{\frac{2a^2}{DT}}-1)+\frac{2|a|e^{\frac{2a^2}{DT}}}D\int_a^\infty e^{-\frac{2x^2}{TD}}dx, \ a\in\mathbb{R}.
\end{aligned}
\end{equation}
\end{proposition}
\begin{proof}By symmetry, it suffices to consider $a>0$.
Let $W(\cdot)$ be a Brownian motion with diffusion coefficient $D$ starting from the origin (the generator of the process is $\frac D2\frac{d^2}{dx^2}$), and denote probabilities and expectations  for this process by $P_0$ and $E_0$.
It is well-known \cite{KS,BO} that  $\tau_a$, the hitting time of $a\in\mathbb{R}$, satisfies
\begin{equation*}
P_0(\tau_a<T|W(T)=0)=e^{-\frac{2a^2}{DT}}.
\end{equation*}
It follows
from the definition of the Brownian bridge and the definition of $X^{\text{bb},1;T}(\cdot)$  that the above equation is equivalent to
\begin{equation}\label{maxBB}
P_0^{\text{bb},1;T}(\tau_a<T)=e^{-\frac{2a^2}{DT}}.
\end{equation}
It is very well-known from the reflection principle  \cite{KS,RW} that the hitting time $\tau_a$ for $W(\cdot)$ has a density $f(t)$ given by
\begin{equation}\label{Reflection}
f(t)=\frac{1}{\sqrt{2\pi D}}\frac{a}{t^\frac32}e^{-\frac{a^2}{2Dt}},\ t>0.
\end{equation}
Of course, the density of $W(t)$ is $\frac{e^{-\frac{a^2}{2Dt}}}{\sqrt{2\pi Dt}},\ a\in\mathbb{R}$.
Thus, $\tau_a$, the hitting time for $W(t)$, $0\le t\le T$, conditioned on $W(T)=0$, or equivalently, the hitting time under $P_0^{\text{bb},1;T}$,  has sub-density
\begin{equation}\label{Reflectionapp}
\frac{f(t)\frac1{\sqrt{2\pi D(T-t)}}e^{-\frac{a^2}{2D(T-t)}}}{\frac1{\sqrt{2\pi DT}}}=
\frac{ae^{-\frac{a^2}{2Dt(1-\frac tT)}}}{\sqrt{2\pi D(1-\frac tT)}\thinspace t^\frac32},\ 0<t<T.
\end{equation}
Consequently, from the definition of $X^{\text{bb},1;T}(\cdot)$,
\begin{equation}\label{condexp}
E_0^{\text{bb},1;T}(\tau_a1_{\tau_a<T})=E_0(\tau_a1_{\tau_a<T}|W(T)=0)=
a\int_0^T\frac{e^{-\frac{a^2}{2Dt(1-\frac tT)}}}{\sqrt{2\pi Dt(1-\frac tT)}}dt.
\end{equation}

From \eqref{maxBB} and the definition of $X^{\text{bb},1;T}(\cdot)$, it follows that
\begin{equation}\label{btwnnn+1}
P_0^{\text{bb},1;T}\big(\tau_a\in(nT,(n+1)T]\big)=(1-e^{-\frac{2a^2}{DT}})^ne^{-\frac{2a^2}{DT}}.
\end{equation}
  Also, from \eqref{maxBB}, \eqref{condexp} and the definition of $X^{\text{bb},1;T}(\cdot)$ it follows that
\begin{equation}\label{expbtwn}
\begin{aligned}
&E_0^{\text{bb},1;T}\big(\tau_a|\tau_a\in(nT,(n+1)T]\big)=nT+E_0^{\text{bb},1;T}(\tau_a|\tau_a<T)=\\
&nT+\frac{E_0^{\text{bb},1;T}(\tau_a1_{\tau_a<T})}{P_0^{\text{bb},1;T}(\tau_a<T)}=
nT+ae^{\frac{2a^2}{DT}}\int_0^T\frac{e^{-\frac{a^2}{2Dt(1-\frac tT)}}}{\sqrt{2\pi Dt(1-\frac tT)}}dt.
\end{aligned}
\end{equation}
Now \eqref{btwnnn+1} and \eqref{expbtwn} yield
$$
\begin{aligned}
&E_0^{\text{bb},1;T}\tau_a=\sum_{n=0}^\infty E_0^{\text{bb},1;T}\big(\tau_a|\tau_a\in(nT,(n+1)T]\big)P_0^{\text{bb},1;T}\big(\tau_a\in(nT,(n+1)T]\big)=\\
&\sum_{n=0}^\infty \Big(nT+ae^{\frac{2a^2}{DT}}\int_0^T\frac{e^{-\frac{a^2}{2Dt(1-\frac tT)}}}{\sqrt{2\pi Dt(1-\frac tT)}}dt\Big)(1-e^{-\frac{2a^2}{DT}})^ne^{-\frac{2a^2}{DT}}=\\
&T(e^{\frac{2a^2}{DT}}-1)+ae^{\frac{2a^2}{DT}}\int_0^T\frac{e^{-\frac{a^2}{2Dt(1-\frac tT)}}}{\sqrt{2\pi Dt(1-\frac tT)}}dt,
\end{aligned}
$$
which gives the first equality in \eqref{expectonept}.  For the second equality in \eqref{expectonept}, we need to prove that
\begin{equation}\label{2ndequal}
\int_0^T\frac{e^{-\frac{a^2}{2Dt(1-\frac tT)}}}{\sqrt{2\pi Dt(1-\frac tT)}}dt=
\frac2D\int_a^\infty e^{-\frac{2x^2}{TD}}dx.
\end{equation}
Let
$G(a)=\int_0^T\frac{e^{-\frac{a^2}{2Dt(1-\frac tT)}}}{\sqrt{2\pi Dt(1-\frac tT)}}dt$.
Then $G'(a)=-\frac a{D\sqrt{2\pi D}}\int_0^T\frac{e^{-\frac{a^2}{2Dt(1-\frac tT)}}}{(t(1-\frac tT))^\frac32}dt$.
The integral appearing in $G'(a)$ is evaluated in \eqref{integral3232} in the course of the proof of Theorem
\ref{expectationa3d}. This gives
\begin{equation}\label{Gprime}
G'(a)=-\frac a{D\sqrt{2\pi D}}\frac{2\sqrt{2\pi D}}ae^{-\frac{2a^2}{TD}}=-\frac2De^{-\frac{2a^2}{TD}}.
\end{equation}
Since $\lim_{a\to\infty}G(a)=0$, we have
$$
G(a)=-\int_a^\infty G'(x)dx=\frac2D\int_a^\infty e^{-\frac{2x^2}{TD}}dx,
$$
proving \eqref{2ndequal}.
\end{proof}

For the  three-dimensional case, we have the following result,  giving upper and lower bounds which depend on  $\epsilon_0$.
\begin{proposition}\label{expectationa3d}
\begin{equation}\label{expectonept3d}
\begin{aligned}
& E_0^{\text{bb},3;T}\tau_a\le T\Big(\frac{|a|+\epsilon_0}{|a|-\epsilon_0}\thinspace\frac{|a|}{2\epsilon_0}e^{\frac{2(|a|+\epsilon_0)^2}{DT}}-1\Big)+
T\frac{|a|+\epsilon_0}{2(|a|-\epsilon_0)}e^{\frac{8|a|\epsilon_0}{DT}},\ |a|>\epsilon_0;\\
&E_0^{\text{bb},3;T}\tau_a\ge T\Big(\frac{|a|}{2\epsilon_0}e^{\frac{2(|a|-\epsilon_0)^2}{DT}}-1\Big)+
T\frac{|a|-\epsilon_0}{2(|a|+\epsilon_0)}e^{-\frac{8|a|\epsilon_0}{DT}},\ |a|>\epsilon_0.
\end{aligned}
\end{equation}
Also,
\begin{equation}\label{unifbdd}
\sup_{\epsilon_0,a: 0<\epsilon_0<|a|\le 1}\epsilon_0E_0^{\text{bb},3;T}\tau_a<\infty.
\end{equation}
\end{proposition}
\bf \noindent Remark.\rm\ We note that \eqref{unifbdd} is a technical result that will be needed because  the right hand side of the first line in \eqref{expectonept3d} is unbounded as $|a|\to\epsilon_0$.
\begin{proof}
Let $W(t)$ be a three-dimensional Brownian motion with diffusion coefficient $D$ and denote probabilities for the process starting from $b\in\mathbb{R}^3$ by   $P_b$.
Abusing notation, let
$P_0(\tau_a=t)$
denote the density of the distribution of $\tau_a$ under $P_0$; we will make similar abuses of notation for other densities in the sequel.
Of course by isotropy, $P_0(\tau_a=t)=P_a(\tau_0=t)$, and this latter density can be found in \cite{HM} in the case that the diffusion coefficient $D$ in equal to one. After appropriate scaling to take into account $D$, this  gives
\begin{equation}\label{free3hitting}
P_0(\tau_a=t)=\frac{\epsilon_0}{|a|}\frac{|a|-\epsilon_0}{\sqrt{2\pi D}\thinspace t^{\frac32}}e^{-\frac{(|a|-\epsilon_0)^2}{2Dt}},\ |a|>\epsilon_0.
\end{equation}
In the sequel, in all formulas involving $a$ and $\epsilon_0$, it will be tacitly assumed that $|a|>\epsilon_0$.
From
\eqref{free3hitting}, the strong Markov property and the definition of $X^{\text{bb},3,T}(\cdot)$, we have
\begin{equation}\label{SMP}
P_0^{\text{bb},3;T}(\tau_a=t)=\frac{P_0(\tau_a=t)\int_{\nu\in\mathbb{R}^3:|\nu|=1}P_{a+\epsilon_0\nu}(W(T-t)=0)\mu_{a,t}(d\nu)}{P_0(W(T)=0)},
\end{equation}
where $\mu_{a,t}$ is the  distribution under $P_0$ of  $W(\tau_a)$, conditioned on $\{\tau_a=t\}$.
Of course,
\begin{equation}\label{Gaussden}
P_b(W(t)=a)=\frac{e^{-\frac{|b-a|^2}{2 Dt}}}{(2\pi Dt)^\frac32},\ a,b\in\mathbb{R}^3.
\end{equation}
Thus, $P_{a+\epsilon_0\nu}(W(T-t)=0)$ attains its maximum and minimum over $\{\nu\in\mathbb{R}^3:|\nu|=1\}$ at $\nu=-\frac{a}{|a|}$ and $\nu=\frac a{|a|}$ respectively, giving
\begin{equation}\label{maxmin}
\frac{e^{-\frac{(|a|+\epsilon_0)^2}{2 D(T-t)}}}{(2\pi D(T-t))^\frac32}\le P_{a+\epsilon_0\nu}(W(T-t)=0)\le \frac{e^{-\frac{(|a|-\epsilon_0)^2}{2 D(T-t)}}}{(2\pi D(T-t))^\frac32},\ |\nu|=1.
\end{equation}
From \eqref{free3hitting}-\eqref{maxmin}, we conclude that
\begin{equation}\label{maxminprob}
\frac{\epsilon_0}{|a|}\frac{|a|-\epsilon_0}{\sqrt{2\pi D}}\frac{e^{-\frac{(|a|+\epsilon_0)^2}{2 D(T-t)}}e^{-\frac{(|a|-\epsilon_0)^2}{2 Dt}}}{( t(1-\frac tT))^\frac32}\le P_0^{\text{bb},3;T}(\tau_a=t)\le
\frac{\epsilon_0}{|a|}\frac{|a|-\epsilon_0}{\sqrt{2\pi D}}\frac{e^{-\frac{(|a|-\epsilon_0)^2}{2 Dt(1-\frac tT)}}}{( t(1-\frac tT))^\frac32}.
\end{equation}
We will use the lower bound in \eqref{maxminprob} to prove \eqref{unifbdd}, however for the proof of the lower bound in
\eqref{expectonept3d}, we will use the following lower bound, which is obtained by replacing the term $|a|-\epsilon$ in the exponent on the left hand side of \eqref{maxminprob} by
$|a|+\epsilon$:
\begin{equation}\label{otherlower}
\frac{\epsilon_0}{|a|}\frac{|a|-\epsilon_0}{\sqrt{2\pi D}}\frac{e^{-\frac{(|a|+\epsilon_0)^2}{2 Dt(1-\frac tT)}}}{( t(1-\frac tT))^\frac32}\le
P_0^{\text{bb},3;T}(\tau_a=t).
\end{equation}

From \eqref{maxBB} and \eqref{Reflectionapp}, it follows that
\begin{equation}\label{integralform}
\frac1{\sqrt{2\pi D}}\int_0^Tb\frac{e^{-\frac{b^2}{2Dt(1-\frac tT)}}}{(1-\frac tT)^\frac12 t^\frac32}dt=e^{-\frac{2b^2}{TD}},\ b>0.
\end{equation}
Making the change of variables $s=T-t$ gives
\begin{equation}\label{COV}
\int_0^T\frac{e^{-\frac{b^2}{2Dt(1-\frac tT)}}}{(1-\frac tT)^\frac12 t^\frac32}dt=\int_0^T\frac{e^{-\frac{b^2}{2D(T-s)\frac sT}}}
{(\frac sT)^\frac12(T-s)^\frac32}ds=\frac1T\int_0^T\frac{e^{-\frac{b^2}{2Ds(1-\frac sT)}}}{s^\frac12(1-\frac sT)^\frac32}ds.
\end{equation}
From
 \eqref{integralform} and \eqref{COV}, we have
\begin{equation}\label{integralformagain}
\frac1{\sqrt{2\pi D}}\int_0^Tb\frac{e^{-\frac{b^2}{2Dt(1-\frac tT)}}}{t^\frac12 (1-\frac tT)^\frac32}dt=Te^{-\frac{2b^2}{TD}}.
\end{equation}
Now \eqref{maxminprob} and \eqref{integralformagain} give
\begin{equation}\label{expectationindicator}
\begin{aligned}
&E_0^{\text{bb},3;T}\tau_a1_{\tau_a\le T}\le
 \frac{\epsilon_0}{|a|}\frac{|a|-\epsilon_0}{\sqrt{2\pi D}}\int_0^T
\frac{e^{-\frac{(|a|-\epsilon_0)^2}{2 Dt(1-\frac tT)}}}{t^\frac12(1-\frac tT)^\frac32}dt=
 \frac{\epsilon_0}{|a|}Te^{-\frac{2(|a|-\epsilon_0)^2}{TD}};\\
& E_0^{\text{bb},3;T}\tau_a1_{\tau_a\le T}\ge
 \frac{\epsilon_0}{|a|}\frac{|a|-\epsilon_0}{\sqrt{2\pi D}}\int_0^T
\frac{e^{-\frac{(|a|+\epsilon_0)^2}{2 Dt(1-\frac tT)}}}{t^\frac12(1-\frac tT)^\frac32}dt=
 \frac{\epsilon_0}{|a|}\frac{|a|-\epsilon_0}{|a|+\epsilon_0} Te^{-\frac{2(|a|+\epsilon_0)^2}{TD}}.
\end{aligned}
\end{equation}

From \eqref{maxminprob} and \eqref{otherlower}, we have
\begin{equation}\label{ratioterm}
\frac{\epsilon_0}{|a|}\frac{|a|-\epsilon_0}{\sqrt{2\pi D}}\int_0^T\frac{e^{-\frac{(|a|+\epsilon_0)^2}{2 Dt(1-\frac tT)}}}{( t(1-\frac tT))^\frac32}dt\le
P_0^{\text{bb},3;T}(\tau_a\le T)\le\frac{\epsilon_0}{|a|}\frac{|a|-\epsilon_0}{\sqrt{2\pi D}}\int_0^T\frac{e^{-\frac{(|a|-\epsilon_0)^2}{2 Dt(1-\frac tT)}}}{( t(1-\frac tT))^\frac32}dt.
\end{equation}
We now show that
\begin{equation}\label{integral3232}
\int_0^T\frac{e^{-\frac{b^2}{2 Dt(1-\frac tT)}}}{( t(1-\frac tT))^\frac32}dt=\frac{2\sqrt{2\pi D}}be^{-\frac{2b^2}{TD}}, \ b>0.
\end{equation}
Indeed, \eqref{integral3232} follows from the following calculation, in which we use \eqref{integralform} for the first equality and \eqref{integralformagain} for the last inequality.
$$
\begin{aligned}
&\frac{\sqrt{2\pi D}}be^{-\frac{2b^2}{TD}}=\int_0^T\frac{e^{-\frac{b^2}{2 Dt(1-\frac tT)}}}{(1-\frac tT)^\frac12t^\frac32}dt=
\int_0^T\frac{e^{-\frac{b^2}{2 Dt(1-\frac tT)}}}{(t(1-\frac tT))^\frac32}(1-\frac tT)dt=\\
&\int_0^T\frac{e^{-\frac{b^2}{2 Dt(1-\frac tT)}}}{(t(1-\frac tT))^\frac32}dt-\frac1T\int_0^T\frac{e^{-\frac{b^2}{2 Dt(1-\frac tT)}}}{t^\frac12(1-\frac tT)^\frac32}dt=
\int_0^T\frac{e^{-\frac{b^2}{2 Dt(1-\frac tT)}}}{(t(1-\frac tT))^\frac32}dt-\frac{\sqrt{2\pi D}}be^{-\frac{2b^2}{TD}}.
\end{aligned}
$$
From    \eqref{ratioterm} and \eqref{integral3232}, we have
\begin{equation}\label{ratiotermagain}
\frac{2\epsilon_0}{|a|}\frac{|a|-\epsilon_0}{|a|+\epsilon_0}e^{-\frac{2(|a|+\epsilon_0)^2}{TD}}\le P_0^{\text{bb},3;T}(\tau_a\le T)\le\frac{2\epsilon_0}{|a|}e^{-\frac{2(|a|-\epsilon_0)^2}{TD}}.
\end{equation}

We now write
\begin{equation}\label{expectationtaua}
E_0^{\text{bb},3;T}\tau_a=\sum_{n=0}^\infty E_0^{\text{bb},3;T}(\tau_a|\tau_a\in(nT,(n+1)T])P_0^{\text{bb},3;T}(\tau_a\in(nT,(n+1)T]).
\end{equation}
From the definition of $X^{\text{bb},3;T}$, we have
\begin{equation}\label{condnT}
 E_0^{\text{bb},3;T}(\tau_a|\tau_a\in(nT,(n+1)T])=nT+ E_0^{\text{bb},3;T}(\tau_a|\tau_a\le T)=nT+\frac{E_0^{\text{bb},3;T}\tau_a1_{\tau_a\le T}}{P_0^{\text{bb},3;T}(\tau_a\le T)}
\end{equation}
and
\begin{equation}\label{geomprob}
P_0^{\text{bb},3;T}(\tau_a\in(nT,(n+1)T])=\big(P_0^{\text{bb},3;T}(\tau_a>T)\big)^nP_0^{\text{bb},3;T}(\tau_a\le T).
\end{equation}
Since for $q\in(0,1)$, one has $\sum_{n=0}^\infty n(1-q)^nq=\frac {1-q}q$, which is decreasing in $q$,
we have from  \eqref{ratiotermagain}
\begin{equation}\label{geometricexp}
\begin{aligned}
&\sum_{n=0}^\infty n\big(P_0^{\text{bb},3;T}(\tau_a>T)\big)^nP_0^{\text{bb},3;T}(\tau\le T)\le
\frac{|a|(|a|+\epsilon_0)}{2\epsilon_0(|a|-\epsilon_0)}e^{\frac{2(|a|+\epsilon_0)^2}{TD}}-1;\\
&\sum_{n=0}^\infty n\big(P_0^{\text{bb},3;T}(\tau_a>T)\big)^nP_0^{\text{bb},3;T}(\tau\le T)\ge
\frac{|a|}{2\epsilon_0}e^{\frac{2(|a|-\epsilon_0)^2}{TD}}-1.
\end{aligned}
\end{equation}
Now  \eqref{expectonept3d} follows from
\eqref{expectationtaua}-\eqref{geometricexp} along with \eqref{ratiotermagain} and \eqref{expectationindicator}.

We now prove \eqref{unifbdd}.
From the upper bound in \eqref{expectonept3d}, it suffices to consider the case that $|a|-\epsilon_0$ is small;
indeed, $|a|-\epsilon_0$ being small is the only thing that can possibly prevent the left hand side of \eqref{unifbdd} from being finite.
Thus we can and  will assume that
\begin{equation}\label{assumpsmalla}
(|a|-\epsilon_0)^2\le \frac T4.
\end{equation}
(Of course, since from \eqref{unifbdd} we are always assuming that  $|a|\le 1$, the assumption \eqref{assumpsmalla} holds automatically if $T\ge4$.)
Let
$$
p(a)=P_0^{\text{bb},3;T}( \tau_a\le T).
$$
From the definition of the Brownian bridge reset search process,
$$
P_0^{\text{bb},3;T}\big( (M-1)T\le \tau_a\le MT\big)=(1-p(a))^{M-1}p(a),\ M\in\mathbb{N}.
$$
Thus,
\begin{equation}\label{expestimate}
\begin{aligned}
&E_0^{\text{bb},3;T}\tau_a\le \sum_{M=1}^\infty MTP_0^{\text{bb},3;T}\big( (M-1)T\le \tau_a\le MT\big)=\\
&Tp(a)\sum_{M=1}^\infty M(1-p(a))^{M-1}=\frac T{p(a)}.
\end{aligned}
\end{equation}

We now obtain a lower bound on $p(a)$.
Using the lower bound in \eqref{maxminprob} for the second inequality below, we have
\begin{equation}\label{paest}
\begin{aligned}
&p(a)=P_0^{\text{bb},3;T}( \tau_a\le T)>P_0^{\text{bb},3;T}( |a-\epsilon_0|^2\le\tau_a\le \frac T2)\ge\\
&\int_{|a-\epsilon_0|^2}^{\frac T2}\frac{\epsilon_0}{|a|}\frac{|a|-\epsilon_0}{\sqrt{2\pi D}}\frac{e^{-\frac{(|a|+\epsilon_0)^2}{2 D(T-t)}}e^{-\frac{(|a|-\epsilon_0)^2}{2 Dt}}}{( t(1-\frac tT))^\frac32}dt\ge\\
&C\epsilon_0\int_{(|a|-\epsilon_0)^2}^\frac T2\frac{|a|-\epsilon_0}{t^\frac32}dt=
2C\epsilon_0\big(1-\frac{\sqrt2(|a|-\epsilon_0)}{\sqrt T}\big)\ge 2C\epsilon_0(1-\frac{\sqrt2}2),\\
&\text{for}\ |a|\le 1 \ \text{and}\ a \ \text{and}\ \epsilon_0\ \text{satisfying}\ \eqref{assumpsmalla},
\ \text{where}\ C\ \text{depends only on}\ T\ \text{and}\ D.
\end{aligned}
\end{equation}
 From  \eqref{expestimate} and \eqref{paest},
we conclude that
$E_0^{\text{bb},3;T}\tau_a\le \frac T{2C\epsilon_0(1-\frac{\sqrt2}2)}$, for $a$ and $\epsilon_0$ as in \eqref{paest}.
This completes the proof of \eqref{unifbdd}.

\end{proof}

\section{Calculating the expected value of $\tau_a$ for fixed $a$ in the periodic  instantaneous reset case}\label{exp-a-inhomog}
We consider the periodic  instantaneous reset process for dimensions $d=1$ and $d=3$.  The formulas in this section (sometimes in slightly different forms) appear, for example,
in \cite{BDR} (for one  and three dimensions), \cite{BBPMC} (for one dimension) and \cite{FBPCM} (for all dimensions).
For completeness, we supply the short proofs.
In the one-dimensional case, we have the following result.
\begin{proposition}\label{expainhomog1}
\begin{equation}\label{expainhomog1form}
E_0^{1;T}\tau_a=T\Big(\big(\int_0^T\frac {|a|}{\sqrt{2\pi D}}\frac1{t^\frac32}e^{-\frac{a^2}{2Dt}}dt\big)^{-1}-1\Big)+\frac{\int_0^T\frac1{t^\frac12}e^{-\frac{a^2}{2Dt}}dt}{\int_0^T\frac1{t^\frac32}e^{-\frac{a^2}{2Dt}}dt},\ a\neq0.
\end{equation}
\end{proposition}
\begin{proof}
By symmetry, it suffices to consider $a>0$. From the definition of the process $X^{1;T}$, the probability of the event $\{\tau_a<T\}$ is the same as it is for a Brownian motion with diffusion coefficient $D$. Thus,
from \eqref{Reflection}, we have
\begin{equation}\label{probless1}
P_0^{1;T}(\tau_a<T)=\int_0^T\frac a{\sqrt{2\pi D}}\frac1{t^\frac32}e^{-\frac{a^2}{2Dt}}dt.
\end{equation}
From \eqref{probless1} and the definition of $X^{1;T}$, we have
\begin{equation}\label{geominhomog1}
P_0^{1;T}(\tau_a\in(nT,(n+1)T])=\big(1-\int_0^T\frac a{\sqrt{2\pi D}}\frac1{t^\frac32}e^{-\frac{a^2}{2Dt}}dt\big)^n\int_0^T\frac a{\sqrt{2\pi D}}\frac1{t^\frac32}e^{-\frac{a^2}{2Dt}}dt,
\end{equation}
and
\begin{equation}\label{condexpinhomog1}
E_0^{1;T}(\tau_a|\tau_a\in(nT,(n+1)T])=nT+\frac{\int_0^T\frac1{t^\frac12}e^{-\frac{a^2}{2Dt}}dt}{\int_0^T\frac1{t^\frac32}e^{-\frac{a^2}{2Dt}}dt}.
\end{equation}
Writing
\begin{equation*}
E_0^{1;T}\tau_a=\sum_{n=0}^\infty E_0^{1;T}\big(\tau_a|\tau_a\in(nT,(n+1)T]\big)P_0^{1;T}\big(\tau_a\in(nT,(n+1)T]\big),
\end{equation*}
the theorem follows from  \eqref{geominhomog1} and \eqref{condexpinhomog1}.

\end{proof}
In the three-dimensional case, we have the following result.
\begin{proposition}\label{expainhomog3}
\begin{equation}\label{expainhomog3form}
E_0^{3;T}\tau_a=T\Big(\big(\int_0^T\frac{\epsilon_0(|a|-\epsilon)}{|a|\sqrt{2\pi D}}\frac1{t^\frac32}e^{-\frac{(|a|-\epsilon_0)^2}{2Dt}}dt\big)^{-1}-1\Big)+\frac{\int_0^T\frac1{t^\frac12}e^{-\frac{(|a|-\epsilon_0)^2}{2Dt}}dt}{\int_0^T\frac1{t^\frac32}e^{-\frac{(|a|-\epsilon_0)^2}{2Dt}}dt},\
|a|>\epsilon_0.
\end{equation}
\end{proposition}
\begin{proof}
From the definition of the process $X^{3;T}$, the probability of the event $\{\tau_a<T\}$ is the same as it is for a three-dimensional Brownian motion with diffusion coefficient $D$. Thus, from \eqref{free3hitting}, we have
\begin{equation}\label{probless3}
P_0^{1;T}(\tau_a<T)=\int_0^T\frac {\epsilon_0(|a|-\epsilon_0)}{|a|\sqrt{2\pi D}}\frac1{t^\frac32}e^{-\frac{(|a|-\epsilon_0)^2}{2Dt}}dt.
\end{equation}
The proof of the proposition now follows exactly as the proof of Proposition \ref{expainhomog1}, but  with \eqref{probless3} replacing \eqref{probless1}.
\end{proof}

\section{Proofs of Theorems \ref{resetGauss}-\ref{inhomog1d}}\label{pfthms1-3}
\noindent \it Proof of Theorem \ref{resetGauss}.\rm\
From \eqref{r}, we have
\begin{equation}\label{Gausscal1}
\begin{aligned}
&\int_{\mathbb{R}}\big(E_0^{(1;r)}\tau_a\big)\mu^{\text{Gauss},1}_{\sigma^2}(da)=2\int_0^\infty\frac{e^{\sqrt{\frac{2r}D}\thinspace a}-1}r\thinspace\frac{e^{-\frac{a^2}{2\sigma^2}}}{\sqrt{2\pi}\sigma}da.
\end{aligned}
\end{equation}
Also,
\begin{equation}\label{Gausscal2}
\begin{aligned}
&\int_0^\infty \frac{e^{\sqrt{\frac{2r}D}\thinspace a}e^{-\frac{a^2}{2\sigma^2}}}{\sqrt{2\pi}\sigma}da=
&e^{\frac{r\sigma^2}D}\int_0^\infty \frac{e^{-\frac{(a-\sqrt{\frac{2r}D}\sigma^2)^2}{2\sigma^2}}}{\sqrt{2\pi}\sigma}da=e^{\frac{r\sigma^2}D}\int_{-\sqrt{\frac{2r}D}\sigma}^\infty \frac{e^{-\frac{x^2}2}}{\sqrt{2\pi}}dx.
\end{aligned}
\end{equation}
We obtain \eqref{expGauss} from \eqref{Gausscal1} and \eqref{Gausscal2}. A change of variables in \eqref{expGauss} yields \eqref{equiv}.
Finally, \eqref{graphcalc} was obtained from \eqref{equiv} using the Desmos graphing calculator.
\hfill $\square$

\medskip

\noindent \it Proof of Theorem \ref{bridgeGauss}.\rm\
From \eqref{expectonept}, we have
\begin{equation}\label{first}
\begin{aligned}
&\int_{\mathbb{R}}\big(E_0^{\text{bb},1;T}\tau_a\big)\mu^{\text{Gauss},1}_{\sigma^2}(da)=2T\int_0^\infty (e^{\frac{2a^2}{DT}}-1)\frac{e^{-\frac{a^2}{2\sigma^2}}}{\sqrt{2\pi}\sigma}da+\\
&2\int_0^\infty ae^{\frac{2a^2}{DT}}\Big(\int_0^T\frac{e^{-\frac{a^2}{2Dt(1-\frac tT)}}}{\sqrt{2\pi Dt(1-\frac tT)}}dt\Big)\frac{e^{-\frac{a^2}{2\sigma^2}}}{\sqrt{2\pi}\sigma}da.
\end{aligned}
\end{equation}
The first integral on the right hand side of \eqref{first} is  infinite if $T\le\frac{4\sigma^2}D$. From now on, we assume that
$T>\frac{4\sigma^2}D$.
We have
\begin{equation*}
2\int_0^\infty e^{\frac{2a^2}{DT}}e^{-\frac{a^2}{2\sigma^2}}da=2\int_0^\infty e^{-\frac12\frac{DT-4\sigma^2}{DT\sigma^2}a^2}da=
\sqrt{2\pi\frac{DT\sigma^2}{DT-4\sigma^2}}.
\end{equation*}
Thus, the first term on the right hand side of \eqref{first} satisfies
\begin{equation}\label{firstpiece}
2T\int_0^\infty (e^{\frac{2a^2}{DT}}-1)\frac{e^{-\frac{a^2}{2\sigma^2}}}{\sqrt{2\pi}\sigma}da=T\big(\frac{DT}{DT-4\sigma^2}\big)^\frac12-T.
\end{equation}

We now turn to the second term on the right hand side of \eqref{first}.
We have
\begin{equation}\label{expint}
\begin{aligned}
&\int_0^\infty ae^{\frac{2a^2}{DT}}e^{-\frac{a^2}{2Dt(1-\frac tT)}}e^{-\frac{a^2}{2\sigma^2}}da=\int_0^\infty ae^{-\frac12\frac{T^2\sigma^2+t(T-t)(DT-4\sigma^2)}{DtT(T-t)\sigma^2}a^2}da=\\
&\frac{DtT(T-t)\sigma^2}{T^2\sigma^2+t(T-t)(DT-4\sigma^2)}.
\end{aligned}
\end{equation}
Using \eqref{expint}, we can write  the second term on the right hand side of \eqref{first} as
\begin{equation}\label{second1}
\begin{aligned}
&2\int_0^\infty ae^{\frac{2a^2}{DT}}\Big(\int_0^T\frac{e^{-\frac{a^2}{2Dt(1-\frac tT)}}}{\sqrt{2\pi Dt(1-\frac tT)}}dt\Big)\frac{e^{-\frac{a^2}{2\sigma^2}}}{\sqrt{2\pi}\sigma}da=\\
&\frac{T^\frac32\sigma}\pi\int_0^T\frac{\sqrt{Dt(T-t)}}{T^2\sigma^2+t(T-t)(DT-4\sigma^2)}dt=
\frac{2T^\frac32\sigma}\pi\int_0^{\frac T2}\frac{\sqrt{Dt(T-t)}}{T^2\sigma^2+t(T-t)(DT-4\sigma^2)}dt.
\end{aligned}
\end{equation}
Make the substitution $x=\sqrt{t(T-t)}$. Then
$t=\frac12\big(T-(T^2-4x^2)^\frac12\big)$ and $dt=2x(T^2-4x^2)^{-\frac12}dx$.  We obtain
\begin{equation}\label{second2}
\int_0^{\frac T2}\frac{\sqrt{Dt(T-t)}}{T^2\sigma^2+t(T-t)(DT-4\sigma^2)}dt=2\sqrt D\int_0^{\frac T2}
\frac1{(T^2-4x^2)^\frac12}\Big(\frac{x^2}{T^2\sigma^2+x^2(DT-4\sigma^2)}\Big)dx.
\end{equation}
Now make the substitution $x=\frac T2\sin\theta$. Then $dx=\frac T2\cos\theta d\theta$. We obtain
\begin{equation}\label{second3}
\int_0^{\frac T2}\frac1{(T^2-4x^2)^\frac12}\Big(\frac{x^2}{T^2\sigma^2+x^2(DT-4\sigma^2)}\Big)dx=
\frac18\int_0^{\frac\pi2}\frac{\sin^2\theta}{\sigma^2+\frac{DT-4\sigma^2}4\sin^2\theta}d\theta.
\end{equation}
We write
$$
\begin{aligned}
&\frac{\sin^2\theta}{\sigma^2+\frac{DT-4\sigma^2}4\sin^2\theta}=\frac4{DT-4\sigma^2}\thinspace\frac{\sin^2\theta}{\sin^2\theta+\frac{4\sigma^2}{DT-4\sigma^2}}=\\
&\frac4{DT-4\sigma^2}-\frac{16\sigma^2}{(DT-4\sigma^2)^2}\thinspace\frac1{\sin^2\theta+\frac{4\sigma^2}{DT-4\sigma^2}}.
\end{aligned}
$$
Thus,
\begin{equation}\label{second4}
\int_0^{\frac\pi2}\frac{\sin^2\theta}{\sigma^2+\frac{DT-4\sigma^2}4\sin^2\theta}d\theta=\frac{2\pi}{DT-4\sigma^2}-\frac{16\sigma^2}{(DT-4\sigma^2)^2}
\int_0^\frac\pi2 \frac1{\sin^2\theta+\frac{4\sigma^2}{DT-4\sigma^2}}d\theta.
\end{equation}
Making the substitution $\tan\theta=s$, in which case $\sin\theta=\frac s{\sqrt{1+s^2}}$ and $d\theta=\frac1{1+s^2}ds$, we obtain for any $A>0$,
\begin{equation}\label{second5}
\begin{aligned}
&\int_0^\frac\pi2\frac1{\sin^2\theta+A}d\theta=\int_0^\infty\frac1{A+(A+1)s^2}ds=\frac1{A+1}\sqrt{\frac{A+1}A}
\arctan\sqrt{\frac{A+1}A}s\Big|_0^\infty=\\
&\frac\pi{2\sqrt{A(A+1)}}.
\end{aligned}
\end{equation}
From \eqref{second5}, we have
\begin{equation}\label{second6}
\begin{aligned}
&\int_0^\frac\pi2 \frac1{\sin^2\theta+\frac{4\sigma^2}{DT-4\sigma^2}}d\theta=\frac\pi2\Big(\frac{4\sigma^2}{DT-4\sigma^2}
\big(\frac{4\sigma^2}{DT-4\sigma^2}+1\big)\Big)^{-\frac12}=\\
&\frac{\pi(DT-4\sigma^2)}{4\sigma\sqrt{DT}}.
\end{aligned}
\end{equation}
From \eqref{second1}-\eqref{second4} and \eqref{second6}, we obtain
\begin{equation}\label{secondpiece}
\begin{aligned}
&2\int_0^\infty ae^{\frac{2a^2}{DT}}\Big(\int_0^T\frac{e^{-\frac{a^2}{2Dt(1-\frac tT)}}}{\sqrt{2\pi Dt(1-\frac tT)}}dt\Big)\frac{e^{-\frac{a^2}{2\sigma^2}}}{\sqrt{2\pi}\sigma}da=\\
&(\frac{2T^\frac32\sigma}\pi)(2\sqrt D)(\frac18)\Big(\frac{2\pi}{DT-4\sigma^2}-\frac{16\sigma^2}{(DT-4\sigma^2)^2}\frac{\pi(DT-4\sigma^2)}{4\sigma\sqrt{DT}}\Big)=\\
&\frac{T^\frac32\sigma\sqrt D}{2\pi}\frac{2\pi\sqrt{DT}-4\sigma\pi}{(DT-4\sigma^2)\sqrt{DT}}=\frac{T^\frac32\sigma\sqrt D}{2\pi}\frac{2\pi(\sqrt{DT}-2\sigma)}{(DT-4\sigma^2)\sqrt{DT}}=\frac{T\sigma}{\sqrt{DT}+2\sigma}.
\end{aligned}
\end{equation}
Now \eqref{expbridgeGauss} follows from \eqref{first}, \eqref{firstpiece} and \eqref{secondpiece}.
A change of variables in \eqref{expbridgeGauss}  yields \eqref{equivagain}. Finally, \eqref{graphcalcagain} was obtained from \eqref{equivagain} using the Desmos graphing calculator.
\hfill $\square$

\medskip

\noindent \it\ Proof of Theorem \ref{inhomog1d}.\rm\  We begin by showing that the left hand side of
\eqref{expGaussinhomog1} is in fact finite if and only if $T>\frac{\sigma^2}D$. We leave it to the reader to show that the second term on the right hand side of
\eqref{expainhomog1form} is bounded as $a\to\infty$. (Make the change of variables, $s=\frac t{a^2}$, then let $x=\frac1{a^2}$ and apply l'H\^opital's rule appropriately.)
Thus, this term is integrable against any Gaussian measure. Now consider the expression
$(\int_0^T\frac a{\sqrt{2\pi D}}\frac1{t^\frac32}e^{-\frac{a^2}{2Dt}}dt\big)^{-1}$ in the first term on the right hand side of \eqref{expainhomog1form}.
It is known from the reflection principle \cite{KS,RW} that
\begin{equation}\label{reflection2ways}
\int_0^T\frac a{\sqrt{2\pi D}}\frac1{t^\frac32}e^{-\frac{a^2}{2Dt}}dt=\int_a^\infty \frac{e^{-\frac{x^2}{2DT}}}{\sqrt{2\pi DT}}dx=
2\int_{\frac a{\sqrt{DT}}}^\infty \frac{e^{-\frac{y^2}2}}{\sqrt{2\pi}}dy.
\end{equation}
It is well-known and can be proved by l'H\^opital's rule that
$\int_x^\infty e^{-\frac{y^2}2}dy\sim\frac1x e^{-\frac{x^2}2}$ as $x\to\infty$.
From this it follows
that
$$
(\int_0^T\frac a{\sqrt{2\pi D}}\frac1{t^\frac32}e^{-\frac{a^2}{2Dt}}dt\big)^{-1}\sim\frac{\sqrt{2\pi}\thinspace a}{2\sqrt{DT}}e^{\frac{a^2}{2DT}},\ \text{as}\ a\to\infty.
$$
From these facts, it follows that \eqref{expGaussinhomog1} is finite if and only if $T>\frac{\sigma^2}D$.

From now on, we assume that $T>\frac{\sigma^2}D$.
From \eqref{expainhomog1form}, in order to prove the theorem, we need to evaluate
$\int_0^\infty \big(\int_0^T\frac a{\sqrt{2\pi D}}\frac1{t^\frac32}e^{-\frac{a^2}{2Dt}}dt\big)^{-1}\frac2{\sqrt{2\pi}\thinspace\sigma}e^{-\frac{a^2}{2\sigma^2}}da$
and $\int_0^\infty \Big(\frac{\int_0^T\frac1{t^\frac12}e^{-\frac{a^2}{2Dt}}dt}{\int_0^T\frac1{t^\frac32}e^{-\frac{a^2}{2Dt}}dt}\Big)\frac2{\sqrt{2\pi}\thinspace \sigma}e^{-\frac{a^2}{2\sigma^2}}da$.
Consider the latter term.
Making the substitutions  $a=\sigma x$ and $T=\frac{\sigma^2}D\mathcal{T}$, we have
\begin{equation}\label{subst12}
\int_0^\infty \Big(\frac{\int_0^T\frac1{t^\frac12}e^{-\frac{a^2}{2Dt}}dt}{\int_0^T\frac1{t^\frac32}e^{-\frac{a^2}{2Dt}}dt}\Big)\frac2{\sqrt{2\pi}\thinspace \sigma}e^{-\frac{a^2}{2\sigma^2}}da=
\int_0^\infty\frac{\int_0^{\frac{\sigma^2}D\mathcal{T}}\frac1{t^\frac12}e^{-\frac{\sigma^2x^2}{2Dt}}dt}{\int_0^{\frac{\sigma^2}D\mathcal{T}}\frac1{t^\frac32}e^{-\frac{\sigma^2x^2}{2Dt}}dt}
\frac2{\sqrt{2\pi}}e^{-\frac{x^2}2}dx.
\end{equation}
Making the substitution   $t=\frac{\sigma^2}D\thinspace s$, we have
\begin{equation}\label{subst3}
\int_0^\infty\frac{\int_0^{\frac{\sigma^2}D\mathcal{T}}\frac1{t^\frac12}e^{-\frac{\sigma^2x^2}{2Dt}}dt}{\int_0^{\frac{\sigma^2}D\mathcal{T}}\frac1{t^\frac32}e^{-\frac{\sigma^2x^2}{2Dt}}dt}
\frac2{\sqrt{2\pi}}e^{-\frac{x^2}2}dx=\frac{\sigma^2}D\Big(\frac2{\sqrt{2\pi}}\int_0^\infty \frac{\int_0^\mathcal{T}\frac1{s^\frac12}e^{-\frac{x^2}{2s}}ds}{\int_0^\mathcal{T}\frac1{s^\frac32}e^{-\frac{x^2}{2s}}ds}\thinspace  e^{-\frac{x^2}2}dx\Big).
\end{equation}
Making the same series of substitutions in the other integral that we need to evaluate, we obtain
\begin{equation}\label{subst2ndfinal}
\int_0^\infty \big(\int_0^T\frac a{\sqrt{2\pi D}}\frac1{t^\frac32}e^{-\frac{a^2}{2Dt}}dt\big)^{-1}\frac2{\sqrt{2\pi}\thinspace\sigma}e^{-\frac{a^2}{2\sigma^2}}da=
2\int_0^\infty \frac1{\int_0^\mathcal{T}\frac x{s^\frac32}e^{-\frac{x^2}{2s}}ds}
\thinspace e^{-\frac{x^2}2}dx.
\end{equation}
Now \eqref{equivagainn} follows from
 \eqref{expainhomog1form}, \eqref{subst3} and \eqref{subst2ndfinal}.
 Making the substitution    $s=\frac{Dt}{\sigma^2}$   in \eqref{equivagainn} gives \eqref{expGaussinhomog1}.
We obtained \eqref{graphcalcagainn} from \eqref{equivagainn} using the Desmos graphing calculator.

 \hfill $\square$

\section{Proofs of Theorems \ref{resetGauss3}-\ref{inhomog3d}}\label{pfthms4-6}
\noindent \it\ Proof of Theorem \ref{resetGauss3}.\rm\
From \eqref{3r} and the fact that $E_0^{(3;r)}\tau_a=0$, for $|a|\le\epsilon_0$, we have
\begin{equation}
\epsilon_0\int_{\mathbb{R}^3}\big(E_0^{(3;r)}\tau_a\big)\mu^{\text{Gauss},3}_{\sigma^2}(da)=
\frac1r\int_{a\in\mathbb{R}^3:|a|>\epsilon_0}\big(|a|e^{\sqrt{\frac rD}(|a|-\epsilon_0)}-\epsilon_0)\frac{e^{-\frac{|a|^2}{2\sigma^2}}}{(2\pi\sigma^2)^\frac32}da.
\end{equation}
Thus, the monotone convergence theorem gives
\begin{equation}\label{limitepsilon}
\lim_{\epsilon_0\to0}\epsilon_0\int_{\mathbb{R}^3}\big(E_0^{(3;r)}\tau_a\big)\mu^{\text{Gauss},3}_{\sigma^2}(da)=
\frac1r\int_{\mathbb{R}^3}|a|e^{\sqrt{\frac rD}|a|}\frac{e^{-\frac{|a|^2}{2\sigma^2}}}{(2\pi\sigma^2)^\frac32}da.
\end{equation}
Letting $R=|a|$ and then letting $x=\frac R\sigma$, we have
\begin{equation}\label{2changeofvar}
\begin{aligned}
&\frac1r\int_{\mathbb{R}^3}|a|e^{\sqrt{\frac rD}|a|}\frac{e^{-\frac{|a|^2}{2\sigma^2}}}{(2\pi\sigma^2)^\frac32}da=
\frac1r\int_0^\infty Re^{\sqrt{\frac rD}R}\frac{e^{-\frac{R^2}{2\sigma^2}}}{(2\pi \sigma^2)^\frac32}4\pi R^2dR=\\
&\frac{2\sigma}{\sqrt{2\pi}\thinspace r}\int_0^\infty x^3e^{\sqrt{\frac rD}\thinspace\sigma x}e^{-\frac{x^2}2}dx.
\end{aligned}
\end{equation}
Letting $s=\sigma^2\frac rD$, we obtain
\begin{equation}\label{rschange}
\frac{2\sigma}{\sqrt{2\pi}\thinspace r}\int_0^\infty x^3e^{\sqrt{\frac rD}\thinspace\sigma x}e^{-\frac{x^2}2}dx=
\frac{\sigma^3}D\big(\frac2{\sqrt{2\pi}s}\int_0^\infty x^3e^{\sqrt sx}e^{-\frac{x^2}2}dx\big).
\end{equation}
Now \eqref{resetGauss3form} follows from \eqref{limitepsilon} and \eqref{2changeofvar},
and
\eqref{resetGauss3formagain}
follows from \eqref{limitepsilon}-\eqref{rschange}. Finally, \eqref{graphiccalc3}  follows from \eqref{resetGauss3formagain} using the Desmos graphing calculator.
\hfill $\square$
\medskip

\noindent \it\ Proof of Theorem \ref{bridgeGauss3}.\rm\
From the monotone convergence theorem, it follows that
$\lim_{\epsilon_0\to0}\int_{\{a\in\mathbb{R}^3:|a|>\epsilon_0\}}e^{\frac{2(|a|-\epsilon_0)^2}{DT}}\mu_{\sigma^2}^{\text{Gauss},3}(da)=\infty$, if $T\le \frac{4\sigma^2}D$.
Thus, from the second line of \eqref{expectonept3d}, it follows that
\begin{equation}\label{infinitecase}
\lim_{\epsilon_0\to0}\epsilon_0\int_{\mathbb{R}^3}\big(E_0^{\text{bb},3;T}\tau_a\big)\mu^{\text{Gauss},3}_{\sigma^2}(da)=\infty,\ \text{if}\ T\le \frac{4\sigma^2}D.
\end{equation}

From now on, we assume that $T>\frac{4\sigma^2}D$.
From \eqref{expectonept3d} and \eqref{unifbdd}, we have
\begin{equation}\label{everywhereconv}
\lim_{\epsilon_0\to0}\epsilon_0E_0^{\text{bb},3;T}\tau_a=\frac T2|a|e^{\frac{2|a|^2}{TD}},\ \text{for all}\ 0\neq a\in \mathbb{R}^3.
\end{equation}
From \eqref{everywhereconv} along with the first line in \eqref{expectonept3d} and \eqref{unifbdd}, it follows from the dominated convergence theorem that
\begin{equation}\label{limitinreset3}
\lim_{\epsilon_0\to0}\epsilon_0\int_{\mathbb{R}^3}\big(E_0^{\text{bb},3;T}\tau_a\big)\mu^{\text{Gauss},3}_{\sigma^2}(da)=
\frac T2\int_{\mathbb{R}^3}|a|e^{\frac{2|a|^2}{TD}}\frac{e^{-\frac{|a|^2}{2\sigma^2}}}{(2\pi\sigma^2)^\frac32}da.
\end{equation}
We have $\frac{2|a|^2}{TD}-\frac{|a|^2}{2\sigma^2}=-(\frac{TD-4\sigma^2)}{TD\sigma^2}\frac{|a|^2}2$.
Thus,
\begin{equation}\label{complsq}
\int_{\mathbb{R}^3}|a|e^{\frac{2|a|^2}{TD}}e^{-\frac{|a|^2}{2\sigma^2}}da=\int_{\mathbb{R}^3}|a|^2e^{-\frac{|a|}2(\frac{TD-4\sigma^2}{TD\sigma^2})}da=
\int_0^\infty Re^{-\frac{R^2}2(\frac{TD-4\sigma^2}{TD\sigma^2})}4\pi R^2dR.
\end{equation}
Integrating by parts yields
\begin{equation}\label{byparts}
\int_0^\infty Re^{-\frac{R^2}2(\frac{TD-4\sigma^2}{TD\sigma^2})}4\pi R^2dR=8\pi(\frac{TD\sigma^2}{TD-4\sigma^2})^2,\ T>\frac{4\sigma^2}D.
\end{equation}
Now
\eqref{expbridgeGauss3} follows from  \eqref{infinitecase}-\eqref{byparts}, \eqref{expbridgeGauss3again} follows immediately from \eqref{expbridgeGauss3} and \eqref{graphiccalc3again} is obtained using the Desmos graphing
calculator.
\hfill $\square$

\medskip

\noindent\it Proof of Theorem \ref{inhomog3d}.\rm\
Considering
\eqref{expainhomog3form},
a proof similar to that given for the corresponding result in the one-dimensional case
shows that the left hand side of  \eqref{expGaussinhomog3d}
is finite if and only if $T>\frac{\sigma^2}D$.

From now on, we assume that $T>\frac{\sigma^2}D$.
Using the dominated convergence theorem for the first term on the right hand side of \eqref{expainhomog3form} and the bounded convergence theorem for the second term there,
and recalling that $E_0^{3;T}\tau_a=0$, for $|a|\le\epsilon_0$, we obtain
\begin{equation}\label{limitepsiloninhomog3}
\lim_{\epsilon_0\to0}\epsilon_0\int_{\mathbb{R}^3}\big(E_0^{(3;T)}\tau_a\big)\mu^{\text{Gauss},3}_{\sigma^2}(da)=T\int_{\mathbb{R}^3}\frac1{\int_0^T\frac1{\sqrt{2\pi D}}\frac1{t^\frac32}e^{-\frac{|a|^2}{2Dt}}dt}\thinspace
\frac{e^{-\frac{|a|^2}{2\sigma^2}}}{(2\pi\sigma^2)^\frac32}da.
\end{equation}
We have
\begin{equation}\label{spherical}
\int_{\mathbb{R}^3}\frac1{\int_0^T\frac1{\sqrt{2\pi D}}\frac1{t^\frac32}e^{-\frac{|a|^2}{2Dt}}dt}\thinspace
\frac{e^{-\frac{|a|^2}{2\sigma^2}}}{(2\pi\sigma^2)^\frac32}da=\int_0^\infty\frac1{\int_0^T\frac1{\sqrt{2\pi D}}\frac1{t^\frac32}e^{-\frac{R^2}{2Dt}}dt}\thinspace
\frac{e^{-\frac{R^2}{2\sigma^2}}}{(2\pi\sigma^2)^\frac32}4\pi R^2dR.
\end{equation}
Making the substitutions $R=x\sigma$ and $T=\frac{\sigma^2}D\mathcal{T}$, we have
\begin{equation}\label{aTsubst}
\int_0^\infty\frac1{\int_0^T\frac1{\sqrt{2\pi D}}\frac1{t^\frac32}e^{-\frac{R^2}{2Dt}}dt}\thinspace
\frac{e^{-\frac{R^2}{2\sigma^2}}}{(2\pi\sigma^2)^\frac32}4\pi R^2dR=
\int_0^\infty\frac1{\int_0^{\frac{\sigma^2}D\mathcal{T}}\frac1{\sqrt{2\pi D}}\frac1{t^\frac32}e^{-\frac{\sigma^2x^2}{2Dt}}dt}\thinspace
\frac{e^{-\frac{x^2}2}}{(2\pi)^\frac32}4\pi x^2dx.
\end{equation}
Making the substitution $t=\frac{\sigma^2}Ds$, we have
\begin{equation}\label{tssubst}
\int_0^\infty\frac1{\int_0^{\frac{\sigma^2}D\mathcal{T}}\frac1{\sqrt{2\pi D}}\frac1{t^\frac32}e^{-\frac{\sigma^2x^2}{2Dt}}dt}\thinspace
\frac{e^{-\frac{x^2}2}}{(2\pi)^\frac32}4\pi x^2dx=2\sigma\int_0^\infty\frac1{\int_0^\mathcal{T}\frac1{s^\frac32}e^{-\frac{x^2}{2s}}ds}\thinspace x^2e^{-\frac{x^2}2}dx.
\end{equation}
From \eqref{limitepsiloninhomog3}-\eqref{tssubst}, we obtain
\eqref{equivinhomog3d}.
Making the substitution $s=\frac{Dt}{\sigma^2}$ in \eqref{equivinhomog3d} gives  \eqref{expGaussinhomog3d}.
We obtained \eqref{graphiccalc3againn} from \eqref{equivinhomog3d} using the Desmos graphic calculator.

\hfill $\square$
\section{Proof of Theorem \ref{2dimcase}}\label{pfthm7}
From \eqref{dr}, we have
\begin{equation}\label{2da}
E_0^{2;r}\tau_a=\frac1r\Big(\frac{K_0(\sqrt{\frac rD}\thinspace\epsilon_0)}{K_0(\sqrt{\frac rD}\thinspace |a|)}-1\Big),\ |a|>\epsilon_0.
\end{equation}
One has the asymptotic formula $K_0(x)\sim-\log\frac x2$ as $x\to0$ \cite[p.80]{Watson}.
Thus,
\begin{equation}\label{logeps}
\lim_{\epsilon_0\to0}\frac1{|\log\epsilon_0|}E_0^{2;r}\tau_a=
\frac1{rK_0(\sqrt{\frac rD}\thinspace |a|)}.
\end{equation}
From the above asymptotic behavior, $\frac1{K_0(\sqrt{\frac rD}\thinspace |a|)}$ is integrable in a neighborhood of  $0\in\mathbb{R}$. Also, $K_0(x)$ decays exponentially as $x\to\infty$ \cite[p.202]{Watson}. Therefore,
 $\frac1{K_0(\sqrt{\frac rD}\thinspace |a|)}$ is integrable against any Gaussian density.
 Using these facts with \eqref{logeps} and the dominated convergence theorem, and recalling that $E_0^{2;r}\tau_a=0$, for $|a|\le\epsilon_0$, we obtain
 \begin{equation}\label{2dlimit}
 \lim_{\epsilon\to0}\frac1{|\log\epsilon_0|}\int_{\mathbb{R}^2}(E^{2;r}\tau_a)\mu_{\sigma^2}^{\text{Gauss},2}(da)=\frac1r\int_{\mathbb{R}^2}\frac1{K_0(\sqrt{\frac rD}\thinspace |a|)}\thinspace\frac{e^{-\frac{|a|^2}{2\sigma^2}}}{2\pi\sigma^2}da.
 \end{equation}
 We have
 \begin{equation}\label{polar}
 \int_{\mathbb{R}^2}\frac1{K_0(\sqrt{\frac rD}\thinspace |a|)}\thinspace\frac{e^{-\frac{|a|^2}{2\sigma^2}}}{2\pi\sigma^2}da=\int_0^\infty\frac1
 {K_0(\sqrt{\frac rD}\thinspace R)}\thinspace\frac{e^{-\frac{R^2}{2\sigma^2}}}{\sigma^2} RdR.
 \end{equation}
 Making the change of variables, $R=\sigma x$, we have
 \begin{equation}\label{substax2d}
 \int_0^\infty\frac1
 {K_0(\sqrt{\frac rD}\thinspace R)}\thinspace\frac{e^{-\frac{R^2}{2\sigma^2}}}{\sigma^2} RdR=\int_0^\infty\frac1{K_0(\sqrt{\frac rD}\thinspace\sigma x)}xe^{-\frac{x^2}2}dx.
\end{equation}
Now \eqref{2dimformula} follows from \eqref{2dlimit}--\eqref{substax2d}.
One obtains \eqref{equiv2dim} from \eqref{2dimformula} by substituting $r=\frac D{\sigma^2}s$.
We obtained \eqref{graphcalc2} from \eqref{equiv2dim}, using the Desmos graphing calculator and  using the representation
$K_0(x)=\int_0^\infty e^{-x\cosh t}dt$ \cite[p.181]{Watson}.
\hfill$\square$

\bf\noindent Acknowledgement.\rm\ The authors thanks two referees whose suggestions and constructive criticism very significantly improved this paper.

\end{document}